\documentclass[12pt]{article}

\usepackage{array, amssymb, amsmath, graphics}

\newtheorem{theorem}{Theorem}
\newtheorem{lemma}{Lemma}

\newtheorem{proposition}{Proposition}

\newtheorem{corollary}{Corollary}
\newtheorem{question}{Question}

\newtheorem{definition}{Definition}
\newenvironment{proof}{{\bf Proof.}}{\hfill\rule{2mm}{2mm}}
\newtheorem{remarka}{Remark}

\def\newpic#1{}

\newlength{\cellwid}
\makeatletter
{\endarray}

\title{Measure preserving homomorphisms and independent sets in tensor graph powers}

\author
{%
Babak Behsaz
\thanks{Department of Computer Engineering and
Information Technology, Amirkabir University of Technology (Tehran
Polytechnic), 424 Hafez Ave., Tehran, Iran.
Email:behsaz@ce.aut.ac.ir.}
\and
Pooya Hatami\thanks{Department of Mathematical Sciences, Sharif
University of Technology, P.O.Box 11365-9415, Azadi Ave., Tehran,
Iran. Email:p\_hatami@ce.sharif.edu.}
}%

\date{}

\begin{document}

\maketitle
\begin{abstract}
In this note, we study the behavior of independent sets of maximum
probability measure in tensor graph powers. To do this, we
introduce an upper bound using measure preserving homomorphisms.
This work extends some previous results about independence ratios
of tensor graph powers.

\end{abstract}

\section{Introduction}
The graphs in this note can have infinite number of vertices. A
\textit{homomorphism} from a graph $H$ to a graph $G$ is a map $h$
from the vertices of $H$ to the vertices of $G$ such that
$h(u)h(v)$ is an edge in $G$ for every edge $uv \in E(H)$. For
every graph $G$, we assume that there is a \textit{probability
measure} $\mu_G$ on the vertices of $G$. A homomorphism
$h:V(H)\rightarrow V(G)$ is \emph{measure preserving}, if $h$ is
measurable and for every measurable $S \subseteq V(G)$,
$\mu_H(h^{-1}(S))=\mu_G(S)$. By $H\rightarrow G$, we mean that
there exists a measure preserving homomorphism from $H$ to $G$.

\begin{definition}\label{def:vt}
Let $G$ be a graph with the probability measure $\mu_G$ on its
vertices. We call $G$ \emph{vertex transitive}, if
\begin{enumerate}
\item there exists a set $S$ of measure preserving homomorphisms
$\phi:V(G) \rightarrow V(G)$;

\item there exists a probability measure $\nu$ on $S$, such that
for almost every $v \in V(G)$, $\phi(v)$ has the same distribution
as $\mu_G$ when $\phi$ is chosen according to $\nu$.
\end{enumerate}
\end{definition}

Note that for a finite graph with the uniform measure, this
definition coincides with the known definition of vertex
transitivity of finite graphs (take $S$ to be the group of
automorphisms of $G$ with the uniform measure).

The \textit{tensor} product of two graphs, $G$ and $H$, has vertex
set $V(G)\times V(H)$, where $(u,v)$ and $(u',v')$ are adjacent if
and only if $uu'\in E(G)$ and $vv'\in E(H)$. The measure on the
new vertex set is the product measure. The characteristics of
tensor products of graphs have been studied extensively (for
example see \cite{brown96,zhu92}).

Let $G^n$ be the tensor product of $n$ copies of $G$. For a graph
$G$, define $\overline{\alpha}(G):= \sup_I {\mu_G}(I)$, where $I$
is a measurable independent set. It is easy to see that if
$H\rightarrow G$, then $\overline{\alpha}(H)
\ge\overline{\alpha}(G)$ and $H^n \rightarrow G^n$. Since $G^{i+1}
\rightarrow G^i$, this in particular implies that
$\overline{\alpha}(G^n)$ is a nondecreasing sequence, and $\lim_{n
\rightarrow \infty} \overline{\alpha}(G^n)$ exists. For a finite
vertex transitive graph $H$ with the uniform measure, it is known
that $\overline{\alpha}(H^n)= \overline{\alpha}(H)$
(see~\cite{alon04}). Now we prove an infinite version of this
fact:

\begin{lemma}
Let $H$ be a (possibly infinite) vertex transitive graph. Then for
any positive integer $n$,
$$\overline{\alpha}(H^n)= \overline{\alpha}(H).$$
\end{lemma}
\begin{proof}
Since $\overline{\alpha}(H^n) \ge \overline{\alpha}(H)$, it is
enough to prove that $\overline{\alpha}(H^n) \le
\overline{\alpha}(H)$. According to Definition~\ref{def:vt}, there
exists a probability measure $\nu$ on a set $S$ that together they
satisfy Definition~\ref{def:vt} (property $1$ and $2$). Consider
an arbitrary measurable independent set $I \subseteq H^n$ and for
a vertex $w \in H^n$ denote by $[w \in I]$ the function that is
$1$ if $w \in I$ and $0$ otherwise. Note that
\begin{equation*}
\mu_{H^n}(I) = \Pr_{v_i \in V(H)}[(v_1,\ldots,v_n) \in
I]=\Pr_{\phi_i \in S, v \in V(H)} [(\phi_1(v),\ldots,\phi_n(v))
\in I]
\end{equation*}
Thus, there exists a choice of
$\overline{\phi}_1,\ldots,\overline{\phi}_n$ such that
\begin{align*}
\mu_{H^n}(I) &\le \Pr_{v \in V(H)}
[(\overline{\phi}_1(v),\ldots,\overline{\phi}_n(v)) \in I]\\ &=
\mu(\{v: (\overline{\phi}_1(v),\ldots,\overline{\phi}_n(v)) \in I,
v\in V(H)\}).
\end{align*}
But $\{v: (\overline{\phi}_1(v),\ldots,\overline{\phi}_n(v)) \in
I\}$ is an independent set in $H$ because $I$ is an independent
set and $\{\overline{\phi}_i\}$ are homomorphisms. Thus we obtain
that $\mu_{H^n}(I) \le \overline{\alpha}(H)$ which completes the
proof.
\end{proof}

We call a vertex transitive graph $H$, a \textit{descriptor} of
$G$, if $H \rightarrow G$. Thus, for a descriptor $H$, we have
$$\overline{\alpha}(H)= \lim_{n \rightarrow
\infty}\overline{\alpha}(H^n)\ge\lim_{n \rightarrow
\infty}\overline{\alpha}(G^n).$$ Now, define $\mathrm{u}(G)$ as
below:
\begin{equation*}
\mathrm{u}(G)=\inf_{descriptor H}\overline{\alpha}(H).
\end{equation*}
Trivially, we have
\begin{equation}\label{ine:main}
\lim_{n \rightarrow \infty}\overline{\alpha}(G^n)\le\mathrm{u}(G).
\end{equation}

This raises the following question:
\begin{question}\label{que:main}
Does every finite graph $G$ satisfy $\lim\overline{\alpha}(G^n)
=\mathrm{u}(G)$?
\end{question}

This question is inspired by the work of Dinur and
Friedgut~\cite{friedgut05}, in which measure preserving
homomorphisms are used to give a new proof for an
Erd\"{o}s-Ko-Rado-type theorem.  We study the behavior of
$\lim\overline{\alpha}(G^n)$ for graphs with probability measures.
This is closely related and can be considered as the
generalizations of some results in~\cite{brown96}
and~\cite{alon06}.

\section{The results}
The following lemma is the generalization of a result of
\cite{brown96} to graphs with probability measures.
\begin{lemma}\label{lem:basic}
For every finite graph $G$, if $\lim \overline{\alpha}(G^n)
> \frac{1}{2}$, then $\lim \overline{\alpha}(G^n) = 1$.
\end{lemma}
\begin{proof}
If $\lim \overline{\alpha}(G^n)>\frac{1}{2}$, then there exists a
positive integer $i$ such that
$\overline{\alpha}(G^i)>\frac{1}{2}$. By letting $H=G^i$,
trivially $\lim \overline{\alpha}(H^n)=\lim
\overline{\alpha}(G^n)$. Let $I$ be an independent set of measure
$\frac{1}{2}+\epsilon$ of $H$. Define $J\subseteq V(H^n)$ as the
set of vertices with strictly more than half coordinates in $I$.
Clearly, $J$ is an independent set of $H^n$. To prove that
$\overline{\alpha}(H^n) = 1$, it suffices to prove that as $n$
goes to infinity a random vertex which is taken from $H^n$ with
respect to $\mu_{H^n}$ is in $J$ almost surely. Let $X_i$ be an
indicator random variable, such that $X_i=1$ if the $i$th
coordinate of the random vertex belongs to $I$ and $X_i=0$
otherwise. As a result, we have $E[X_i] = \overline{\alpha}(H)$
and the mean and variance of $X_i$ is finite. Thus, by applying
the weak law of large numbers for the random variable $X =
\frac{1}{n}\sum_{i=1}^n X_i$, we obtain
$\lim_{n\rightarrow\infty}P(| X-\overline{\alpha}(H) | <
\epsilon')= 1$ for every positive real $\epsilon'$. Therefore, $X$
is greater than $\frac{1}{2}$ almost surely as desired.
\end{proof}

Now, we characterize the graphs for which $\lim
\overline{\alpha}(G^n) = 1$ and by using this, we present some
classes of graphs satisfying $\lim\overline{\alpha}(G^n)
=\mathrm{u}(G)$.
\begin{lemma}\label{lem:necessity}
For every finite graph $G$, if $\mathrm{u}(G)=1$ then there exists
an independent set $I\subseteq V(G)$ such that $\mu_G(I) >
\mu_G(N(I))$, where $N(I)$ is the set of the vertices in $V(G)$
that are adjacent to at least one vertex in $I$.
\end{lemma}
\begin{proof}
Suppose that every independent set $I\subseteq V(G)$ satisfies
$\mu_G(I)\le \mu_G(N(I))$. We claim that for all $Q \subseteq
V(G)$, we have $\mu_G(Q) \le \mu_G(N(Q))$. Suppose that for a $Q
\subseteq V(G)$, we have $\mu_G(Q)
> \mu_G(N(Q))$. Let $I$ be the set of all vertices of $Q$ without
any neighbor in $Q$. Clearly, $I$ is an independent set and since
$\mu_G(Q) > \mu_G(N(Q))$, $I$ is nonempty. Let $Q'=Q \backslash
I$. Hence, $Q' \subseteq N(Q)$ and $N(I)\subseteq N(Q) \backslash
Q'$. Therefore,
$$\mu_G(N(I))\le \mu_G(N(Q)) - \mu_G(Q') <
\mu_G(Q) - \mu_G(Q') = \mu_G(I),$$ a contradiction.

Now let $G' = G\times K_2$, where $K_2=uv$ has the uniform
measure. It is clear that $X = \{(z,u)\in V(G'): z \in V(G)\}$ and
$Y = V(G') - X$ is a bipartition of $G'$. Consider a flow network
with vertices $V(G')\cup \{s, t\}$ and nonnegative capacities
$c(s,x) = \mu_{G'}(x)$, and $c(y,t) = \mu_{G'}(y)$, for $x \in X$
and $y\in Y$, and $c(x,y) = \infty$ if $xy \in E(G')$. All the
other capacities are $0$. Let $(S,T)$ be a minimum cut of this
network with capacity $c(S,T)$. By the structure of the flow
network, we have $c(S,T)\le \frac{1}{2}$. Now, let $X_1=S\cap X$,
$Y_1=S\cap Y$, $X_2=T\cap X$, and $Y_2=T\cap Y$. Since
$c(x,y)=\infty$ if $xy\in E(G')$, there is not any edge between
$X_1$ and $Y_2$. Therefore, $X_1\cup Y_2$ is an independent set in
$G'$. Since for all $Q \subseteq V(G)$, $\mu_G(Q) \le
\mu_G(N(Q))$, we have $\mu_{G'}(X_1)\le\mu_{G'}(N(X_1))$ and
$\mu_{G'}(Y_2)\le\mu_{G'}(N(Y_2))$, which yields $\mu_{G'}(X_1) +
\mu_{G'}(Y_2) \le \mu_{G'}(N(X_1))+ \mu_{G'}(N(Y_2))$. Thus, we
obtain $\mu_{G'}(X_1)+\mu_{G'}(Y_2)\le \frac{1}{2}$. Therefore, we
have $\mu_{G'}(X_2)+\mu_{G'}(Y_1)\ge \frac{1}{2}$ and because
$c(S,T)=\mu_{G'}(X_2)+\mu_{G'}(Y_1)$, we obtain
$c(S,T)=\frac{1}{2}$. Thus by the max-flow min-cut theorem, the
value of a maximum flow $f$ must be equal to $\frac{1}{2}$.

Now by using the maximum flow $f$, we construct a descriptor graph
$H$ for $G'$ together with the measure preserving homomorphism
$h:H \rightarrow G'$ as follows. The vertices of $H$ are the
elements of the interval $[0,1)$ endowed with the (uniform)
Lebesgue measure, and $E(H)=\{\{a,a+\frac{1}{2}\}: a \in
[0,\frac{1}{2})\}$. It is easy to see that $H$ is vertex
transitive. Now we have to specify $h$. For $xy \in E(G')$, let
$f_{xy}$ denote the amount of the flow that passes through this
edge. Since the value of $f$ is equal to $\frac{1}{2}$, we have
$\sum_{xy \in E(G')} f_{xy}=\frac{1}{2}$. So it is possible to
partition the interval $[0,\frac{1}{2})$ into \emph{disjoint}
intervals in the following way: $[0,\frac{1}{2})= \bigcup_{xy\in
E(G')} [a_{xy},a_{xy}+f_{xy})$, where $a_{xy} \ge 0$. Now $h$ is
defined as for every $z \in V(G')=X \cup Y$:

$$h^{-1}(z) = \left\{
\begin{array}{lcl}
   \bigcup_{y:\;zy\in E(G')} [a_{zy},a_{zy}+f_{zy}) &\ &\mbox{if $z \in X$} \\
   \bigcup_{x:\;xz\in E(G')} [\frac{1}{2}+a_{xz},\frac{1}{2}+a_{xz}+f_{xz}) && \mbox{if $z \in
   Y$}
\end{array}\right. $$

It is not hard to see that $h$ is a measure preserving
homomorphism from $H$ to $G'$. Since $G' \rightarrow G$, $H$ is a
descriptor of $G$. Hence, we have $\mathrm{u}(G)\le\frac{1}{2}$.
\end{proof}
\begin{lemma}\label{lem:sufficiancy}
For every finite graph $G$, if there exists an independent set
$I\subseteq V(G)$ such that $\mu_G(I) > \mu_G(N(I))$, then $\lim
\overline{\alpha}(G^n) = 1$.
\end{lemma}
\begin{proof}
Let $U = V(G) \backslash (I \cup N(I))$. Let $m_n =
\overline{\alpha}(G^n)$. Trivially, $\mu_G(I)+\mu_G(N(I))+\mu_G(U)
= 1$, $m_1 \ge \mu_G(I)$ and $\mu_G(U)<1$. Consider the union of
the vertices with first coordinate in $I$ and the vertices with
first coordinate in $U$ and last $n-1$ coordinates in the maximum
measure independent set of $G^{n-1}$. It can be seen that this is
an independent set and we have $m_n \ge \mu_G(I) +
\mu_G(U)m_{n-1}$. By applying this inequality repeatedly, we
obtain:
\begin{equation*}\begin{split}
m_n&\ge \mu_G(I)+\mu_G(I)\mu_G(U)+...+\mu_G(U)^{n-1}.m_1  \\
&\ge \mu_G(I)+\mu_G(I)\mu_G(U)+...+\mu_G(I)\mu_G(U)^{n-1} =
\frac{\mu_G(I)-\mu_G(I)\mu_G(U)^n}{1-\mu_G(U)}
\end{split}\end{equation*}
Thus, we have $\lim_{n\rightarrow \infty} m_n \ge
\frac{\mu_G(I)}{1-\mu_G(U)}=
\frac{\mu_G(I)}{\mu_G(I)+\mu_G(N(I))}>\frac{1}{2}$, and by
Lemma~\ref{lem:basic}, we have $\lim \overline{\alpha}(G^n)=1$.
\end{proof}
\begin{theorem}\label{thm:main}
For every finite graph $G$, the followings are equivalent:
\begin{itemize}
\item[(i)] $\lim \overline{\alpha}(G^n) = 1$;

\item[(ii)] $\mathrm{u}(G)=1$;

\item[(iii)] there exists an independent set $I\subseteq V(G)$
such that $\mu_G(I) > \mu_G(N(I))$.
\end{itemize}
\end{theorem}
\begin{proof}
(i) implies (ii) by the inequality (\ref{ine:main}), (ii) implies
(iii) by Lemma~\ref{lem:necessity}, and (iii) implies (i) by
Lemma~\ref{lem:sufficiancy}.
\end{proof}
\begin{corollary}\label{cor:main}
For every finite graph $G$, if $\lim \overline{\alpha}(G^n) \in
\{\frac{1}{2},1\}$ then $\lim
\overline{\alpha}(G^n)=\mathrm{u}(G)$.
\end{corollary}
\begin{remarka}
It is not hard to see that for graphs with rational measures,
Theorem~\ref{thm:main}(i) directly yields
Theorem~\ref{thm:main}(iii). Noga Alon showed us that a density
argument can be used to generalize this to graphs with real
measures~\cite{alon06b}. But, since we are mainly interested in
$\mathrm{u}(G)$, we do not state his proof here.
\end{remarka}

Corollary~\ref{cor:main} presents a family of graphs for which
equality holds in Question~\ref{que:main}. In the next
proposition, we show that this family contains bipartite graphs.
Trivially, finite vertex transitive graphs are another family of
graphs for which equality holds in Question~\ref{que:main}.
\begin{proposition}\label{prp:bipartite}
For a finite bipartite graph $G$, we have $\lim
\overline{\alpha}(G^n)\in\{\frac{1}{2},1\}$.
\end{proposition}
\begin{proof}
Let $X$ and $Y$ be a bipartition of $G$. The set of the vertices
of $G^n$ whose first coordinates are in $X$ and the set of the
vertices of $G^n$ whose first coordinates are in $Y$ is a
bipartition of $G^n$. Thus, for the bipartite graph $G^n$,
$\overline{\alpha}(G^n)\ge \frac{1}{2}$. Therefore, by
Lemma~\ref{lem:basic}, we obtain $\lim \overline{\alpha}(G^n) \in
\{\frac{1}{2},1\}$.
\end{proof}

\large
\parskip = 5 mm
\noindent\textbf{Acknowledgements}

\normalsize \small We are grateful to Noga Alon for reading one of
the first drafts of this paper and pointing out a direct proof to
show that Theorem~\ref{thm:main}(i) yields
Theorem~\ref{thm:main}(iii). In addition, we appreciate Hamed
Hatami for his invaluable comments and helps through preparation
of this paper.
\parskip = 1 mm

\bibliographystyle{plain}
\bibliography{references}

\end{document}